\documentclass[11pt,a4paper]{article}
\usepackage{authblk}
\usepackage{graphicx}
\usepackage{latexsym}
\usepackage{subfigure}
\usepackage{amsmath}
\usepackage{enumerate}
\usepackage{amssymb}
\usepackage[mathscr]{eucal}
\usepackage{amsthm}
\usepackage{bbm}
\usepackage{bm}
\usepackage{color}
\usepackage[small]{caption}
\usepackage[hypertexnames=false,colorlinks=true,linkcolor=blue,citecolor=blue,
            bookmarksopen=false,bookmarks=false,
            pdfstartview=XYZ,pdffitwindow=true,pdfcenterwindow=true]{hyperref}
\usepackage[a4paper,text={6.0in,9.0in},centering,
            includefoot,foot=0.6in]{geometry}
\usepackage[printfigures]{figcaps}

\allowdisplaybreaks[1]

\newtheorem{thm}{Theorem}[section]

\newtheorem{prop}[thm]{Proposition}

\theoremstyle{definition}

\newtheorem{rem}[thm]{Remark} 
\numberwithin{equation}{section}
\renewcommand{\d}{\mathrm{d}}

\newcommand{\Z}{\mathbb{Z}}
\newcommand{\R}{\mathbb{R}}

\newcommand{\T}{\mathbb{T}} 
\newcommand{\J}{\mathcal{J}}
\renewcommand{\P}{\mathcal{P}}  
\newcommand{\V}{\mathcal{V}}
\newcommand{\A}{\mathcal{A}}

\newcommand{\eps}{\epsilon}     
\newcommand{\diag}{\,\mathrm{diag}} 
\begin{document}
 
\figcapsoff
 
\title{Asymptotic-preserving projective integration schemes for kinetic
  equations in the diffusion limit}

\author[1]{Pauline Lafitte} \author[2]{Giovanni Samaey} \affil[1]{ \small
    	Lab. Painlev\'e 
	(CNRS UMR8524 –- Univ.  Lille 1) \&  	SIMPAF (INRIA Lille Nord
        Europe), France
	 } \affil[2]{ \small Department of Computer Science,
  K.U. Leuven, Celestijnenlaan 200A, B-3001 Leuven, Belgium }
  
\date{May 20, 2010}

\maketitle

\begin{abstract}
  We investigate a projective integration scheme for a kinetic equation in the
  limit of vanishing mean free path, in which the kinetic description
  approaches a diffusion phenomenon. The scheme first takes a few small steps
  with a simple, explicit method, such as a spatial centered flux/forward
  Euler time integration, and subsequently projects the results forward in
  time over a large time step on the diffusion time scale. We show
  that, with an appropriate choice of the inner step size, the time-step
  restriction on the outer time step is similar to the stability condition for the
  diffusion equation, whereas the required number of inner steps does not
  depend on the mean free path.  We also provide a consistency result. The
  presented method is asymptotic-preserving, in the sense that the method
  converges to a standard finite volume scheme for the diffusion equation in
  the limit of vanishing mean free path. The analysis is illustrated with
  numerical results, and we present an application to the Su-Olson test.
\end{abstract}

\clearpage
\section{Introduction}
 
In many applications, ranging from radiative transfer over rarefied gas
dynamics to cell motion in biology, the underlying physical system consists of
a large number of moving and colliding particles. Such systems can be
accurately modelled using a kinetic mesoscopic description that governs the
evolution of the particle distribution in position-velocity phase space.  In a
diffusive scaling, when the mean free path of the particles is small with
respect to the (macroscopic) length scale of interest, a macroscopic
description involving only a few low-order moments of the particle
distribution (such as a diffusion equation in neutron transport and radiative
transfer, or fluid equations in rarefied gas dynamics) may give a rough idea
of the behavior.  However, refining the description is uneasy because, whereas
the physical model becomes much simpler in the diffusion limit, a direct
numerical simulation of the kinetic model tends to be prohibitively expensive
due to the additional dimensions in velocity space and stability restrictions
that depend singularly on the mean free path.

We consider dimensionless kinetic equations of the type
\begin{equation}\label{e:intro_kinetic}
  \partial_t f_\eps+\frac{v}{\eps}\cdot\nabla_x f_\eps = \frac{1}{\eps^2}Q(f_\eps),
\end{equation}
modeling the evolution of a particle distribution function $f_\eps(x,v,t)$
that gives the distribution density of particles at a given position $x\in
U\subset\R^d$ with velocity $v\in V\subset\R^d$, $d\geq 1$, at time $t$, the
collisions being embodied in the operator $Q$.  The parameter $\eps>0$ is
meant as the ratio of the mean free path over the characteristic length of
observation, i.e.~the average distance traveled by the particles between
collisions.  The diffusion limit is obtained by taking $\eps\to 0$. Under some
appropriate assumptions, which will be detailed in section \ref{sec:model},
the unknown $f_\eps$ then relaxes on short time-scales to an equilibrium, in
which the dependence on $v$ is fixed, and the dynamics of the system on long
time-scales can be described as a function of the density
$\rho_\eps(x,t)=\langle f_\eps(x,v,t) \rangle$, where
\begin{equation*}
  \langle \cdot \rangle = \int_V\;\cdot\; d\mu(v),
\end{equation*}
is the averaging operator over velocity space and $(V,d\mu)$  denotes
the measured velocity space.  For $\epsilon\to 0$, the density
$\rho=\lim_{\epsilon\to 0}\rho_\eps$ satisfies formally the diffusion equation
\begin{equation}\label{e:diff}
  \partial_t \rho - d \Delta_{x} \rho = 0, \qquad d_p=\langle v^2\rangle.
\end{equation}

The difficulty in studying the asymptotic behavior numerically is precisely
due to the presence of two time-scales. On the one hand, explicit time
integration of equation \eqref{e:intro_kinetic} is numerically challenging
because one is forced to take very small time-steps $\delta t$ when $\eps\to
0$ to stably integrate the fast relaxation. Indeed, due to stability
considerations, $\delta t$ needs to be shrunk as $\eps\to 0$ to properly
satisfy both the $\eps$-dependent hyperbolic Courant-Friedrichs-Lewy (CFL)
condition for equation \eqref{e:intro_kinetic} and the stability constraints
for the collision term. On the other hand, implicit schemes are
computationally expensive because of the extra dimensions in velocity
space. At the same time, the equation closely resembles a diffusion equation
in that limit, for which a parabolic CFL condition of the type $\Delta
t=O(\Delta x^2)$ (independently of $\eps$) would be desirable. 

A number of specialized methods that are \emph{asymptotic-preserving} in the
sense introduced by Jin \cite{jin} have been developed that can integrate
equation \eqref{e:intro_kinetic} in the limit of $\epsilon\to 0$ with
time-steps that are only limited by the stability constraints of the diffusion
limit. We briefly review here some efforts, and refer to the cited references
for more details.  In \cite{jipato,kl98}, separating the distribution $f$ into
its odd and even parts in the velocity variable results in a coupled system of
transport equations where the stiffness appears only in the source term,
allowing to use a time-splitting technique \cite{strang} with implicit
treatment of the source term; see also related work in
\cite{jin,Jin:2000p11547,Klar:1998p11621,Klar:1999p11589,Naldi:1998p11655}.
When the collision operator allows for an explicit computation, an explicit
scheme can be obtained by splitting $f$ into its mean value and the
first-order fluctuations in a Hilbert expansion form \cite{gl1} under a
classical diffusion CFL.  Also, closure by moments \cite[e.g.]{cogogo,le} can
lead to reduced systems for which time-splitting provides new classes of
schemes \cite{CGLV}. Alternatively, a micro-macro decomposition based on a
Chapman-Enskog expansion has been proposed \cite{lemi}, leading to a system of
transport equations that allows to design a semi-implicit scheme without time
splitting.  An innovative non-local procedure based on the quadrature of
kernels obtained through pseudo-differential calculus was proposed in
\cite{BG}.

Our goal is to introduce a different point of view, based on
methods that were developed for large multiscale systems of ODEs.  In \cite{GK}, projective integration methods were introduced
as a class of explicit methods for the solution of stiff systems of ordinary
differential equations (ODEs) that have a large gap between their fast and
slow time scales; these methods fit within recent efforts to systematically
explore numerical methods for multiscale simulation
\cite{EEng03,E:2007p3747,KevrGearHymKevrRunTheo03,Kevrekidis:2009p7484}. In
projective integration, the fast modes, that correspond to the Jacobian
eigenvalues with large negative real parts, decay quickly, whereas the slow
modes correspond to eigenvalues of smaller magnitude and are the solution
components of practical interest.  Such problems are called stiff, and a
standard explicit method requires time steps that are very small compared to
the slow time scales, just to ensure a stable integration of the fast
modes. Projective integration circumvents this problem. The method first takes
a few small (inner) steps with a simple, explicit method, until the transients
corresponding to the fast modes have died out, and subsequently projects
(extrapolates) the solution forward in time over a large (outer) time step; a
schematic representation of the scheme is given in figure
\ref{fig:schematic_pi}.
\begin{figure}[h!]
  \begin{center}
    \includegraphics[scale=0.4]{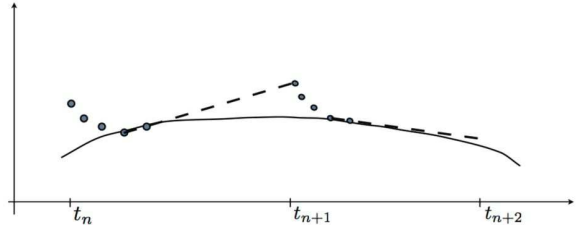}
  \end{center}
  \caption{\label{fig:schematic_pi}Schematic representation of projective
    forward Euler. In a stiff system with a spectral gap, the values of the
    solution obtained by straightforward explicit integration (dots) are
    quickly attracted to a slow manifold (solid curve). Projective integration
    takes a few explicit inner steps until the solution has come down to
    this slow manifold.  Then, a chord slope estimation is performed, and a
    big projective outer step is taken. Since the result of
    the outer step does not lie on the slow manifold, a number of inner steps
    is taken again, and the procedure is repeated.}
\end{figure}
A stability analysis in the ODE setting was presented for a first order
version, called projective forward Euler \cite{GK}, and an accuracy
analysis was given in \cite{Vandekerckhove:2008uq}. Higher-order versions
have been proposed in \cite{Lee2007p2355,RicoGearKevr04}.

Projective integration methods can offer a number of important advantages for
the simulation of kinetic equations. In particular, they are fully explicit
and do not require any splitting, neither in time, nor in microscopic and
macroscopic variables. In this work, we will analyze the properties of
projective integration for kinetic equations on diffusive time scales in one
space dimension, keeping in mind that these methods extend readily to higher
dimensions.  The paper is organized as follows. In section \ref{sec:model}, we
give the necessary preliminaries on the model problem \eqref{e:intro_kinetic}
and present its diffusion asymptotics.  We also introduce the Su-Olson test
case that will be used in the numerical experiments.  In section
\ref{sec:disc}, we develop the numerical scheme.  We present the brute-force
inner schemes and the projective outer forward Euler method. In section
\ref{sec:stab}, we show that, when choosing the inner time step $\delta
t=\epsilon^2$, the stability condition on the outer time step is independent
of $\epsilon$, and similar to the CFL condition of the limiting heat
equation. Moreover, the required number of inner steps is also independent of
$\epsilon$ when $\epsilon\to 0$. Subsequently, in section \ref{sec:consist},
we study the consistency of the method and, using the stability condition
derived in section \ref{sec:stab}, we give a bound of the convergence error,
enabling to conclude that the presented scheme is asymptotic-preserving.  We
provide numerical illustrations for the linear model and the Su-Olson test in
section \ref{sec:num}. Finally, section \ref{sec:concl} contains conclusions
and an outlook to future work.

\section{Model problem\label{sec:model} and diffusion asymptotics}

\subsection{Linear relaxation}

We consider equation \eqref{e:intro_kinetic} in one space dimension,
\begin{equation}\label{e:kinetic}
  \partial_tf_\eps+\dfrac{v}{\eps}\partial_xf_\eps
  =\dfrac{1}{\eps^2}Q(f_\eps),
\end{equation}
and specify the collision operator as a linear relaxation
\begin{equation*}
  Q(f_\eps)=\langle f_\eps\rangle-f_\eps = \rho_\eps-f_\eps
\end{equation*}
that can be interpreted as the difference of a gain and a loss term.  In the
remainder of the text, we restrict the discussion to a periodic setting in one
space dimension, hence $x\in\mathbb{T}=[0,1)$ and $v\in V \subset
\mathbb{R}$. The results of the derivation, however, could easily be
generalized to higher space dimensions.

Throughout the paper, the measured velocity space $(V,\mu)$ is required to
satisfy
\begin{equation*}
  \begin{cases}
    \int_{V}d\mu(v)=1,\\
    \int_{V}h(v)d\mu(v)=0,\qquad \text{for any odd integrable function $h:V\longrightarrow \mathbb{R}$,}\\
    \int_V v^2 d\mu(v)=d>0.
  \end{cases}
\end{equation*}
Typical examples are
\begin{itemize}
\item $V=(-1,1)$ endowed with the normalized Lebesgue measure, for which we
  have $d=1/3$;
\item the discrete velocity space
  \begin{equation}\label{e:disc-vel}
    V=\{-v_p,\ldots,-v_1,v_1,\ldots,v_p\}, \text{ with } v_j=\dfrac{2j-1}{2p},\,\,j\in\J=\J_+\cup\J_-,
  \end{equation} where $\J_\pm:=\{\pm 1,\ldots,\pm p\}$, endowed with the
  normalized discrete velocity measure
  $d\mu(v)=\dfrac{1}{2p}\sum_{j\in\J}\delta(v-v_j)$, for which
  $$d_p=\dfrac{\sum_jv_j^2}{2p}=\dfrac{4p^2-1}{12p^2}$$ so that $d_p\to 1/3$ as $p\to \infty$;
\item $V=\mathbb{R}$ endowed with the Gaussian measure
  $d\mu(v)=(2\pi)^{-1/2}\exp(-v^2/2)dv$, for which $d=1$.
\end{itemize}
These assumptions on the velocity space result in a number of properties for
$Q$, namely :
\begin{enumerate}
\item $Q$ is a bounded operator on $L^p(V,d\mu(v)), 1\le p \le \infty$;
\item $Q$ is conservative, i.e.~
  \[
  \forall f \in L^1(V,d\mu(v)) : \langle Q(f)\rangle = 0;
  \]
\item the elements of the kernel of $Q$ are independent of $v$.
\end{enumerate}

As $\epsilon\to 0$, the frequency of collisions increases; hence, we may
formally propose to write $f_\epsilon$ as a perturbation of the macroscopic
density $\rho_\eps=\langle f_\eps\rangle$ using a Hilbert expansion:
\begin{equation}\label{e:hilb}
  f_\epsilon(x,v,t)=\rho_\eps(x,t)+\epsilon g_\eps(x,v,t). 
\end{equation}
Equation \eqref{e:kinetic} can then alternatively be written as
\begin{equation}\label{e:kin_expanded}
  \begin{cases}
    \partial_t \rho_\eps(x,t)+\langle v\partial_x g_\eps(x,v,t)\rangle=0,\\
    \partial_t g_\eps(x,v,t)+\dfrac{1}{\eps}\left(v\partial_x g_\eps-\langle
      v\partial_x
      g_\eps\rangle\right)=-\dfrac{1}{\eps^2}\left(g_\epsilon+v\partial_x\rho_\eps
    \right).
  \end{cases}
\end{equation}
Taking formally the limit $\eps\to 0$ yields
\begin{equation}
  \label{e:diffeps}
  \partial_t\rho_\eps-\langle v^2\rangle\partial_{xx}\rho_\eps=O(\eps^2),
\end{equation}
since 
\begin{equation}
  \label{e:g}
  g_\eps=-v\partial_x\rho_\eps+\eps\left\{\langle v\partial_x
  g_\eps\rangle-v\partial_x g_\eps\right\}+O(\eps^2).
\end{equation}
The approximation \eqref{e:diffeps} is consistent at order 2 in $\eps$ with
the diffusion equation \eqref{e:diff}.
We refer to \cite{cogogo} for a recent starting point of the literature on the
convergence of $f_\eps$ to $\rho$, solution of \eqref{e:diff}.

For concreteness, we will use the discrete velocity space \eqref{e:disc-vel}
in our analysis and simulations. Thus, we are now interested in a
vector-valued function $f_\eps:\R^+\times\T\longrightarrow\R^{2p}$, of which
we denote the component corresponding to $v_j$ as $f_{\eps,j}(t,x)$, $\forall
j \in \mathcal{J}$. In this setting, the density is given as
$\rho_\eps=\left(\sum_{j}f_{\eps,j}\right)/(2p)$.  The kinetic equation
\eqref{e:kinetic} then reads
\begin{equation}
  \label{e:discrete-kinetic}
  \forall j\in\J,\qquad\partial_t f_{\eps,j}+\frac{v_j}{\eps}\partial_x f_{\eps,j} = \frac{\rho_\eps-f_{\eps,j}}{\eps^2}.
\end{equation}

\subsection{Su-Olson equation}

While the linear kinetic equation \eqref{e:kinetic} is an ideal model problem
for analysis purposes, we will also show numerical results for a more
challenging test case, namely the traditional Su-Olson benchmark, a prototype
model for radiative transfer problems.  Here, the unknown $f_\eps$ represents
the specific intensity of radiations, which interact with matter through
energy exchanges; see e.g.~\cite{BG,buet-despres,CGLV,SuOlson}.  A complete
model couples a kinetic equation for the evolution of $f_\eps$ with the Euler
system describing the evolution of the matter. In the Su-Olson test, this
coupling is replaced by a simple ODE describing the evolution of the material
temperature.  The system reads
\begin{equation}\label{e:su-olson}
  \begin{cases}
   & \partial_t f_\eps+\dfrac{v}{\eps}f_\eps =
    \dfrac{1}{\eps^2}\left(\rho_\eps-f_\eps\right)
    +\sigma_a\left(\Theta-\rho_\eps \right)+S,\\
    &\partial_t \Theta = \sigma_a\left(\rho_\eps-\Theta\right).
  \end{cases}
\end{equation}
Here, $\Theta=T^4$, with $T$ the material temperature, and $S$ is a given
source depending on $x$. In our simulations, the parameter $\sigma_a=1$.

\section{Numerical scheme\label{sec:disc}}

\subsection{Finite volume formulation}

We consider a uniform, constant in time, periodic spatial mesh with spacing
$\Delta x$, consisting of cells $C_i=[x_{i-1/2},x_{i+1/2})$, $1\leq i \leq
N_x$ with $N_x\Delta x = 1$ centered in $x_i$, where $x_i=i\Delta x$,
and a uniform mesh in time $T_k=[t^k,t^{k+1})$, $k \ge 0$, $t^k=k\delta t$
where $\delta t$ is the time step. We adopt a finite volume approach, that is
we integrate \eqref{e:discrete-kinetic} on a cell $M_{i,k}=C_i\times T_k$ to
obtain, $\forall j\in\J$,
\begin{equation*}
  \int_{C_i}(f_{\eps,j}(t^{k+1},\cdot)-
  f_{\eps,j}(t^{k},\cdot)) + \dfrac{v_j}{\eps}
  \left(\int_{T_k}(f_{\eps,j}(\cdot,x_{i+1/2})-f_{\eps,j}(\cdot,x_{i-1/2}))
    \right)=\dfrac{1}{\eps^2}\int_{M_{i,k}}\left(\rho_\eps-f_{\eps,j}\right).
\end{equation*}
In order to simplify the notations, the solution of the continuous equation
\eqref{e:discrete-kinetic} $f_\eps$ will always be denoted with the subscript
$\eps$ whereas the solution of the numerical scheme will be denoted
$f=(f^k_{i,j})$.
This leads to the conservative forward
Euler scheme
\begin{align}
  f_{i,j}^{k+1}&=f_{i,j}^{k}-\dfrac{\delta t}{\eps\Delta
    x}\left(\phi(f)_{i+1/2,j}^k-\phi(f)_{i-1/2,j}^k\right)+\dfrac{\delta
    t}{\eps^2}\left(\rho^k_{i}-f^k_{i,j}\right),\label{e:num_scheme}
\end{align}
where $f_{i,j}^k$ denotes an approximation of the mean value
$\int_{C_i}f_{\eps,j}(t^k,x)\d x$, the numerical flux $\phi(f)_{i+1/2,j}^k$ is
an approximation of the flux at the interface $x_{i+1/2}$ at time $t^k$ in the
equation for the velocity $j$, and $\rho^k_i=\langle f_i^k\rangle$, the
average being taken over the velocity index.
We will consider upwind or centered numerical fluxes that are given by
\begin{eqnarray}
  && \phi_u(f)^k_{i+1/2,j}=\begin{cases}v_j f^k_{i,j},&
    \text{ if } j\in\J_+ \;\;(v_j>0) ,\\
    v_j f^k_{i+1,j}, &\text{ if } j\in\J_- \;\;(v_j<0), \end{cases}\label{e:upwind_fe}\\
  && \phi_c(f)^k_{i+1/2,j}=v_j\dfrac{f^k_{i+1,j}+f^k_{i,j}}{2},\label{e:central_fe}
\end{eqnarray}
respectively.  We introduce a generic short-hand notation,
\begin{equation}\label{e:sdt}
  f^{k+1}=S_{\delta t}f^k,
\end{equation}
with $f\in\R^{N_x\times 2p}$ and $S_{\delta t}$ a square matrix of order
$N_x\times 2p$.
\begin{rem}[Maximum principle] We recall that, while the centered flux scheme
  \eqref{e:central_fe} is second-order accurate in space, it does not obey a
  maximum principle, and hence may lead to unphysical oscillations, the
  centered transport scheme being violently unstable for transport
  equations. Any projective integration scheme based on the centered flux
  scheme should therefore also violate the maximum principle. However, since
  the kinetic equation \eqref{e:kinetic} is consistent at order 2 in $\eps$
  with a diffusion equation \eqref{e:diffeps}, the oscillations are quickly
  stabilized  as $\eps\to0$ (see also the discussion on consistency in section
  \ref{sec:consist}).
\end{rem}

\subsection{Projective integration}

Because of the presence of the small parameter $\eps$, the time steps that one
can take with the upwind scheme are at most $O(\eps)$, due to the CFL
stability condition for the transport equation, or $O(\eps^2)$ due to the
relaxation term.  However, in the diffusion limit, as $\eps$ goes to $0$, the
equation tends to the diffusion equation \eqref{e:diff}, for which a standard
finite volume/forward Euler method only needs to satisfy a stability
restriction of the form $\Delta t \le \Delta x^2/(2d)$.

In this paper, we consider the use of projective integration \cite{GK}
to accelerate brute force integration; the idea is the following (see also
figure~\ref{fig:schematic_pi}).  Starting from an approximate solution $f^N$
at time $t^N=N\Delta t$, one first takes $K+1$ \emph{inner} steps of size
$\delta t$,
\begin{equation*}
  f^{N,k+1}=S_{\delta t}f^{N,k}, \qquad k=0,\ldots,K,	
\end{equation*}
in which the superscript pair $(N,k)$ represents an approximation to the
solution at $t^{N,k}=N\Delta t +k\delta t$. The aim is to obtain a discrete
derivative to be used in the \emph{outer}
step to compute $f^{N+1}=f^{N+1,0}$ via extrapolation in time, e.g.,
\begin{equation}\label{e:pfe}
  f^{N+1} = f^{N,K+1} + \left(\Delta t-(K+1)\delta t\right)\frac{f^{N,K+1} -
    f^{N,K}}{\delta t}.
\end{equation}
This method is called projective forward Euler, and it is the simplest
instantiation of this class of integration methods
\cite{GK}. Adams--Bashforth or Runge--Kutta extensions of
\eqref{e:pfe}, giving a higher order consistency in terms of $\Delta t$, are
possible \cite{Lee2007p2355,RicoGearKevr04}.

Projective integration is a viable asymptotic-preserving scheme if, as $\eps$
goes to $0$, we have (i) a stability for the outer time step $\Delta t$ that
should satisfy a condition similar to the CFL condition for the diffusion
equation, (ii) a number of inner steps that is independent of $\epsilon$ and
(iii) the consistency with the diffusion equation \eqref{e:diff}. The analysis
of these properties will be performed in the next sections.

\section{Stability analysis\label{sec:stab}}

\subsection{Notations}

To perform a Von Neumann analysis of the projective forward Euler scheme, we
need the following notations :
\begin{itemize}
\item $e:=\dfrac{1}{\sqrt{2p}}(1,\hdots,1)^T \in \R^{2p}$ and $\P:=ee^T$ is
  the orthogonal projection on $\mbox{Span}(e)$;
\item $\forall W\in\R^{2p}$, $\P W=\langle W\rangle
  (1,\hdots,1)^T=\sqrt{2p}\langle W\rangle e$, with $\displaystyle\langle
  W\rangle=\dfrac{1}{2p}\sum_{j=1}^{2p} x_j$,
\item for all $\zeta=\xi\Delta x$, $\xi\in\R$,
  $\V_C(\zeta):=\dfrac{\sin(\zeta)}{\Delta x}\diag(v_p
  ,\hdots,v_1,-v_1,\hdots,-v_p)$ and

  $\V_U(\zeta):=\dfrac{2\sin(\zeta/2)}{\Delta x}\diag(v_p
  e^{i\zeta/2},\hdots,v_1 e^{i\zeta/2},-v_1 e^{-i\zeta/2},\hdots,-v_p
  e^{-i\zeta/2})$.
\end{itemize}
We also introduce the symbol $\mathcal{D}\left(\alpha,\beta\right)$ to denote
a closed disk with center $\alpha$ and radius $\beta$.

\subsection{Forward Euler schemes}

Let us first locate the spectrum of the matrix $S_{\delta t}$ defined in
\eqref{e:sdt}.  We denote $h=S_{\delta t}f$ and compute the Fourier series
in space of periodized reconstructions of $f$ and $h$ as constant-by-cell
functions $F:x\mapsto f_{i} \mbox{ if }x\in C_i$ and $H:x\mapsto h_{i} \mbox{
  if }x\in C_i$ : $\forall m\in\Z$,
\begin{equation}\label{e:A}
  \hat{H}(m)=\A\hat{F}(m)=\left(\left(1-\dfrac{\delta t}{\eps^2}\right)I_{2p}
    +\dfrac{\delta t}{\eps}i\V+\delta t\dfrac{\P}{\eps^2}\right)\hat{F}(m)
\end{equation}
where $\hat{F}(m):=\int_0^1e^{-im2\pi x}f(x)\d x$, and correspondingly
$\hat{H}(m)$, are the $m$-th Fourier coefficients of $F$ and $H$ and $\V$ is
the (diagonal) Fourier matrix of the finite volume operator chosen for the
convection part. We will give hereafter the results for $\V=\V_C$ and
$\V=\V_U$. Since $F\in L^2(0,1)\mapsto (\hat{F}(m))_m\in \ell^2(\Z)$ is an
isometry, studying the stability of the scheme is equivalent to studying the
spectrum of $\A$.

We first prove an auxiliary result.  For the sake of simplicity, the matrix
$\V$ being a diagonal complex matrix, we can study the spectrum of a typical
matrix $A=D+P$ where $P=(1,\hdots,1)^T(1,\hdots,1)$ and $D$ is a diagonal
matrix with complex entries.
\begin{prop}\label{prop:DP}
  Let $D=\diag(D_1,\hdots,D_{2p})$, with $D_j\in \mathbb{C}$. Then the
  following properties of $A=D+P$ hold :
  \begin{itemize}
  \item[\emph{(P1)}] if  $D_j= D_k$ implies $j=k$  \emph{(H1)}, the eigenspaces of $A$ are all of dimension~1 and the
    spectrum of $A$ does not contain any $D_k$~;
  \item[\emph{(P2)}] if moreover we assume that $\overline{D_j}=D_{2p-j+1}$
    $\forall j\in \J_+$ \emph{(H2)}, then an eigenvalue of $A$ is
    \begin{itemize} \item either real
    \item or complex-conjugate and lies in one of the disks of diameters
      $[D_j,\overline{D_j}],\,j\in \J_+$ (see figure \ref{fig:rad}.a);
      \begin{figure}[h]
        \centering \subfigure[(P2)]{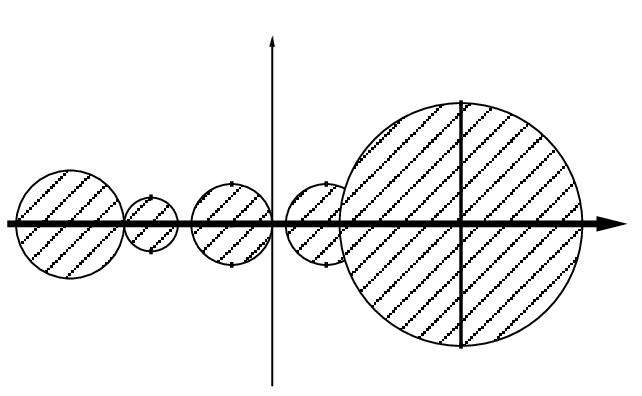 } \qquad
        \subfigure[(P3)]{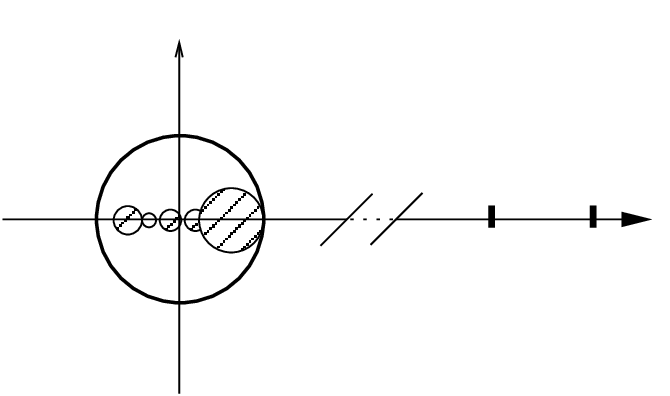 }
        \caption{Spectrum of $A$ in case (P2) and in case (P3) where
          $\mathcal{D}_\eps=\mathcal{D}\left(0,\eps\max_{j\in
              \J^+}(|\alpha_j|+|\beta_j|)\right)$.+}
        \label{fig:rad}
      \end{figure}
    \end{itemize}
  \item[\emph{(P3)}] if, in addition to the previous hypotheses, $D$ is of
    order $\eps$, $\eps$ being small, that is $D=\eps \diag(\alpha_p-i
    \beta_p,\hdots,\alpha_1-i\beta_1,\alpha_1+i\beta_1,\hdots,
    \alpha_{p}+i\beta_{p})$ \emph{(H3)}, then
    \begin{equation*}
      \mbox{\emph{Sp($A$)}}\subset \left(\mathcal{D}\left(0,\eps\max_{j\in
            \J^+}(|\alpha_j|+|\beta_j|)\right)\setminus\R\right) \cup \{\lambda(\eps)\}  
    \end{equation*}
    where the \textbf{only} real eigenvalue $\lambda(\eps)$ is simple and can
    be expanded as
    \begin{equation*}
      \lambda(\eps)=2p\left(1+\eps \dfrac{\langle\alpha\rangle}{2p}
        -\dfrac{\eps^2}{4p^2}\langle(\alpha-\langle\alpha\rangle)^2+\beta^2\rangle\right)+o(\eps^2),
    \end{equation*}
    (see figure \ref{fig:rad}.b).
  \end{itemize}

\end{prop}
\begin{proof}
  Assume (H1).  Let $(\lambda,W)$ be an eigenvalue and an associated
  eigenvector. Then $(D+P)W=\lambda W \Longleftrightarrow (\lambda
  I_{2p}-D)W=\sqrt{2p}\langle W\rangle e$. There are two possibilities~:
  \begin{itemize}
  \item either $\langle W\rangle= 0$ and there exists a pair of indices
    $(k_1,k_2)\in\{1,\hdots,2p\},k_1\neq k_2$ such that $x_{k_1}\neq0$ and
    $x_{k_2}\neq 0$, which implies that $\lambda=D_{k_1}=D_{k_2}$, which is
    incompatible with (H1)~;
  \item or $\langle W\rangle\neq 0$ and $W=\sqrt{2p}\langle W\rangle(\lambda
    I_{2p}-D)^{-1}e$, that is, the eigenspaces of $A$ are all of dimension 1
    and no $D_k$ can be an eigenvalue of $A$.
  \end{itemize}
  The property (P1) is proved.

  Let us have a look at the localization of the eigenvalues of $A$.  The
  characteristic polynomial of $A$ is:
  \begin{equation*}
    \chi_A(\lambda)=\prod_{j=1}^{2p}(D_j-\lambda)
    -\sum_{k=1}^{2p}\prod_{j\neq k}(D_j-\lambda)=Q-Q'
  \end{equation*}
  where $Q=\prod_{j=1}^{2p}(D_j-\lambda)$.\\
  In addition to assuming (H1), assume now (H2) is satisfied. Then $\chi_A$ is
  a real coefficient polynomial, so its roots are either real numbers or
  complex conjugate.  Let $U$ be the union of the disks of diameters
  $[D_j,D_{2p-j}]$, $j\in \J_+$. The property (P2) is analogous to Jensen's
  theorem \cite{HH}. To prove it, let us consider a complex number
  $z\in\mathbb{C}\setminus U$ and an integer $j\in \J_+$. Then
  \begin{eqnarray*}
    \Im\left(\dfrac{1}{z-D_j}+\dfrac{1}{z-\overline{D}_j}\right)  & = & -2
    \Im(z)\dfrac{(\Re(z)-\Re(D_j))^2+\Im(z)^2-\Im(D_j)^2}{|z-D_j|^2|z-\overline{D}_j|^2}\\
    & = & -\Im(z)s_j
  \end{eqnarray*}
  where $s_j>0$ since $z\not\in U$, $\Re(D_j)$ being the center of the disk of
  diameter $[D_j,\overline{D_j}]$ and $\Im(D_j)$ its radius. Now assume $z$ is
  a complex, non-real root of $\chi_A$. Then, taking the imaginary part of the
  equality
  \begin{equation*}
    1= \dfrac{Q'(z)}{Q(z)}=\sum_{j=1}^p \left(\dfrac{1}{z-D_j}+\dfrac{1}{z-\overline{D}_j}\right),
  \end{equation*}
  one gets
  \begin{equation*}
    0= \sum_{j=1}^p (-\Im(z)s_j),
  \end{equation*}
  which is absurd. So (P2) stands for all diagonal ``conjugate'' matrices $D$.

  Finally, assume (H1)-(H2)-(H3) are satisfied, that is we change $D$ into
  $\eps D$. We also order
  $\alpha_1\leq\alpha_2\leq\hdots\leq\alpha_{p-1}\leq\alpha_{p}$.  Let
  $\lambda$ be an eigenvalue of $A$.  If $\eps=0$, $A$ is diagonalizable, its
  eigenvalues are $2p$, of multiplicity $1$, and $0$, of multiplicity
  $2p-1$. According to the theory of perturbations, we want to prove that
  $\lambda$ is necessarily either in a neighborhood of size $O(\eps)$ of the
  origin or in a neighborhood of size $O(\eps)$ of $2p$. Moreover, there
  should be only one eigenvalue, real, in the neighborhood of $2p$.  We
  already know that the non-real eigenvalues of $A$ are located in the closed
  disk $\mathcal{D}(0,\eps\max_{j\in \J_+}(|\alpha_j|+|\beta_j|))$.  Let
  $\lambda$ be a real root of $\chi_A$. Then a simple computation yields
  \begin{eqnarray*}
    \sum_{j=1}^p\left(\dfrac{1}{\lambda-D_j}+\dfrac{1}{\lambda-\overline{D}_j}\right) 
    & = & 2 \sum_{j=1}^p  \dfrac{\lambda-\eps\alpha_j}{|\lambda-\eps D_j|^2}\\
    & = & 1.
  \end{eqnarray*}
  One notes at once that necessarily $\lambda\geq\eps\alpha_1$ in order for
  the sum to be non-negative.  Let us study the behavior of the function
  \begin{equation*}
    h:(\eps,y)\mapsto 2 \sum_{j=1}^p
    \dfrac{y-\eps\alpha_j}{|y-\eps D_j|^2}
  \end{equation*}
  Computing the derivative with respect to $y$, we find that, for $\eps>0$, on
  the interval $(\eps\max_{j\in\J_+}(\alpha_j+|\beta_j|),\infty)$, $h(\eps,\cdot)$ is
  decreasing.  Since for $y>\eps\max_{j\in\J_+}(\alpha_j+|\beta_j|)\geq\eps\alpha_p$,
  $h(\eps,y)>0$ and $\lim_{\infty}h(\eps,\cdot)=0$, there is at most one real
  eigenvalue larger than $\eps \alpha_p$.
  Using the implicit function theorem, and expanding $\eps\mapsto
  h(\eps,\lambda(\eps))-1$ in a neighborhood of $0$, knowing that
  $\lambda(0)=2p$, the only root of $\chi_A$ that is larger than
  $\eps\max_{j\in\J_+}(\alpha_j+|\beta_j|)$ can be expanded
  as $$\lambda(\eps)=2p\left(1+\dfrac{1}{2p}\langle\alpha\rangle\eps
    +\dfrac{1}{4p^2}\langle(\alpha-\langle\alpha\rangle)^2+\beta^2\rangle\eps^2\right)+o(\eps^2),$$
  which concludes the proof of (P3).

\end{proof}

We now turn to the amplification matrix $\A$ in \eqref{e:A} and express it in
terms of $P$ and $D$.  This matrix can indeed be written as
\begin{align*}
  \A  &=\left(1-\dfrac{\delta
      t}{\epsilon^2}\right)I_{2p}+\dfrac{1}{2p}\dfrac{\delta
    t}{\epsilon^2}(P+D),
\end{align*}
with $D=i2p\epsilon\V=\epsilon
\diag(\alpha_p-i\beta_p,\ldots,\alpha_p+i\beta_p)$.

 One can then directly formulate the following proposition :

\begin{prop}\label{prop:spectrum_fe}
  The eigenvalues $\lambda_{\delta t}^j$, $i=j,\ldots, 2p$, of the
  amplification matrix $\A$ defined in \eqref{e:A} are contained in two
  regions: there are $2p-1$ eigenvalues in the disk
  \[
  \mathcal{D}^2=\mathcal{D}\left(1-\dfrac{\delta t}{\eps^2},\dfrac{\delta t
    }{2p\eps}\max_{j\in\J_+}\left(|\alpha_j|+|\beta_j|\right)\right),
  \]
  and one real eigenvalue, an expansion of which is given by
  \[\lambda^1_{\delta t}=1+\dfrac{\langle\alpha\rangle\delta t
  }{2p\eps}-\dfrac{\delta
    t}{4p^2}\langle\alpha^2+\beta^2-\langle\alpha\rangle^2\rangle)+\delta t\ o(1).
  \]
\end{prop}
The proposition is easily verified by inserting the expression for $\A$ into
proposition \ref{prop:DP}.
These eigenvalues can be further examined for the upwind and the centered flux
scheme. 
For the centered flux \eqref{e:central_fe} for which $\V=\V_C(\zeta)$, we have 
\[
\alpha_j=0,\qquad \beta_j= -2p\dfrac{\sin(\zeta)}{\Delta x}v_j,\qquad j\in\J_+,
\] so that
\begin{equation*}
\begin{cases}
  \lambda^1_{\delta t}=1-\dfrac{\delta t}{\Delta x^2}\sin^2(\zeta)\langle
  v^2\rangle + \delta t\; o(1),\\
\\
  \lambda^j_{\delta t}\in \mathcal{D}\left(1-\dfrac{\delta
      t}{\epsilon^2},\dfrac{\delta t}{\epsilon\Delta x}v_p\right), \qquad
  j=2,\ldots,2p,
\end{cases}
\end{equation*}
whereas for the upwind flux \eqref{e:upwind_fe} for which $\V=\V_U(\zeta)$,
\[
\alpha_j=2p\dfrac{2\sin^2(\zeta/2)}{\Delta x}|v_j|,\qquad
\beta_j=2p\dfrac{\sin(\zeta)}{\Delta x}v_j,\qquad j\in\J_+
\]
so that we get, noting that $\langle \alpha\rangle = 4p\sin^2(\zeta/2)\langle |v|\rangle/\Delta x$, 
\begin{equation*}
\begin{cases}
  \lambda^1_{\delta
    t}=1+\dfrac{\delta t}{\epsilon\Delta x}2\sin^2(\zeta/2)\langle
  |v|\rangle-4\dfrac{\delta t}{\Delta x^2}\sin^2(\zeta/2)\left(\langle
    v^2\rangle-\sin^2(\zeta/2)\langle |v|\rangle^2\right)+\delta t\;
  o\left(\dfrac{1}{\epsilon}\right),\\
\\
  \lambda^j_{\delta t}\in \mathcal{D}\left(1-\dfrac{\delta
      t}{\epsilon^2},\dfrac{2\delta t}{\epsilon\Delta x}v_p\right), \qquad
  j=2,\ldots,2p.
\end{cases}
\end{equation*}

We now illustrate this result numerically. We consider equation
\eqref{e:kinetic} on the velocity space \eqref{e:disc-vel} using $p=10$ with
$\epsilon=1\cdot 10^{-2}$ on a mesh $\Pi:=\left\{x_0=-1+\Delta
  x/2,\ldots,1-\Delta x/2\right\}$ with $\Delta x=0.05$ and periodic boundary
conditions. We compute the eigenvalues of a forward Euler time integration
with $\delta t = \epsilon^2$ and $\delta t= 0.5\epsilon^2$, respectively, for
both the upwind and centered flux schemes. The results are shown in figure
\ref{fig:illustr}. Clearly, the spectrum of the forward Euler time-stepper
possesses a spectral gap. The eigenvalues in $\mathcal{D}^2$ correspond to
modes that are quickly damped in the kinetic equation, whereas the eigenvalue
close to $1$ corresponds to the slowly decaying modes that survive on long
(diffusion) time scales.  We see that, for both the upwind and central
schemes, the fast eigenvalues are centered around $1-\delta
  t/\epsilon^2$. The eigenvalues close to $1$ are of order $1-\epsilon$ for
the upwind scheme and of order $1-\epsilon^2$ for the central scheme.
\begin{figure}
  \begin{center}
    \includegraphics[scale=0.75]{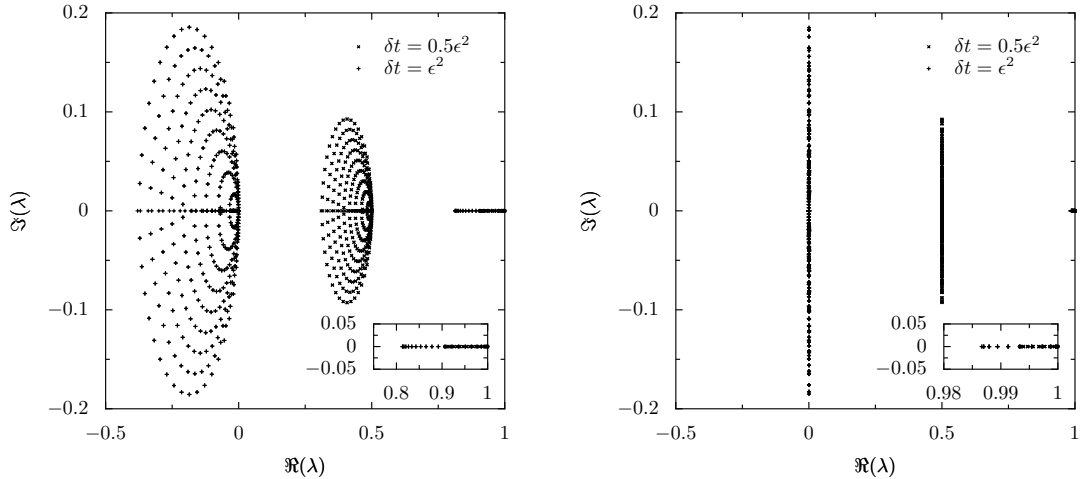}
  \end{center}
  \caption{\label{fig:illustr} Spectrum of a forward Euler time-stepper for a
    spatial finite volume formulation of the kinetic equation
    \eqref{e:kinetic} for different values of $\delta t$. Left: upwind
    scheme. Right: central scheme. The inset shows a zoom of the neighbourhood
    around $1$.  }
\end{figure}

\subsection{Projective integration\label{sec:pi}}

The next step is to examine how the parameters of the projective integration
method need to be chosen to ensure overall stability. It can easily be seen
from \eqref{e:pfe} that the projective forward Euler method is stable if
\begin{equation}\label{e:stab_cond}
  \left|\left[\left(\dfrac{\Delta t-(K+1)\delta t}{\delta
          t}+1\right)\lambda_{\delta t}
      -\dfrac{\Delta t-(K+1)\delta t}{\delta t}\right](\lambda_{\delta t})^K\right|\le 1,
\end{equation}
for all eigenvalues $\lambda_{\delta t}$ of the forward Euler time integration
of the kinetic equation.

The goal is to take a projective time step $\Delta t = O(\Delta x^2)$, whereas
$\delta t=O(\eps^2)$ necessarily to ensure stability of the inner brute-force
forward Euler integration.  Since we are interested in the limit $\eps\to 0$ for fixed $\Delta x$, we look at the limiting stability
regions as $\Delta t/\delta t\to \infty$.  In this regime, it is shown in
\cite{GK} that the values $\lambda_{\delta t}$ for which the condition
\eqref{e:stab_cond} is satisfied lie in two separated regions
$\mathcal{D}^{PI}_1\cup\mathcal{D}^{PI}_2$ which each approaches a disk,
\[
\mathcal{D}^{PI}_1=\mathcal{D}\left(1-\dfrac{\delta t}{\Delta t},\dfrac{\delta
    t}{\Delta t}\right)\text{ and
}\mathcal{D}^{PI}_2=\mathcal{D}\left(0,\left(\dfrac{\delta t}{\Delta
      t}\right)^\frac{1}{K}\right).
\]
The eigenvalues in $\mathcal{D}^{PI}_2$ correspond to modes that are quickly
damped by the time-stepper, whereas the eigenvalues in $\mathcal{D}_1^{PI}$
correspond to slowly decaying modes.  The projective integration method then
allows for accurate integration of the modes in $\mathcal{D}^{PI}_1$ while
maintaining stability for the modes in $\mathcal{D}^{PI}_2$.

Based on the formulae for the eigenvalues $\lambda^j_{\delta t}$ and the
stability regions of projective integration, we are able to determine the
method parameters $\delta t$, $\Delta t$ and $K$. The first observation is
that, to center the fast eigenvalues of the inner time integration (that are
in $\mathcal{D}^2$) around $0$, one should choose $\delta t = \epsilon^2$.
Note that this time step is chosen to ensure a quick damping of the
corresponding modes. The maximal time step that can be taken for stability of
the inner integration would be $\delta t\approx 2\epsilon^2$ due to the bounds
in $\mathcal{D}^2$ ; in that case, however, the fast modes of the kinetic
equations are only slowly damped.

\begin{rem}\label{rem:dx}
  For the choice $\delta t=\epsilon^2$, the spectral properties reveal a
  natural, but important, restriction on the required mesh size $\Delta x$,
  which needs to satisfy $\Delta x \geq v_p\epsilon$, to ensure stability of the inner forward Euler method.
  Therefore, the limit $\Delta x \to 0$ for fixed $\epsilon$ is not considered
  in this text.
\end{rem}

Before deciding on the number of inner forward Euler steps, we need to choose
the projective, outer step size $\Delta t$. To this end, we require the real
eigenvalue of the forward Euler time integration to satisfy $\lambda^1_{\delta
  t}\in\mathcal{D}_1^{PI}$, that is
\[
1-2\dfrac{\delta t}{\Delta t}\le \lambda_{\delta t}^{1}\le 1.
\]
The second inequality is always satisfied.  For the central scheme, we have
\[\lambda^1_{\delta t}=1-\dfrac{\delta t \sum_j v_j^2}{p\Delta x^2}+\delta t\ o(1).\] Using
$\delta t = \epsilon^2$ and $\sum_j v_j^2/(2p)=d_p$, we then obtain
\begin{equation*}
  \Delta t \le 2\dfrac{\Delta x^2}{d_p},
\end{equation*}
which is similar to the CFL condition for a forward Euler time integration of
the heat equation. (Note that the maximal allowed time step is a factor four
larger than that of the heat equation.)

\begin{rem}
  The similar derivation for the upwind scheme shows that, in that case,
  $\Delta t=O(\eps)$, which is undesirable.  We will see further on that there are obstructions in the consistency analysis too.
\end{rem}

Finally, $\Delta x$ being fixed beforehand, we need to determine the number of
small steps $K$.  Introducing $r=\epsilon/\Delta x$ and $\nu=d_p\Delta
t/\Delta x^2$, stability is ensured if the eigenvalues of the
forward Euler time integration that are in $\mathcal{D}^2$ are contained in
the region $\mathcal{D}_2^{PI}$.  This leads to the condition
\[
v_p r \leq \left(\dfrac{d_pr^2}{\nu}\right)^{(1/K)},
\]
which, after some algebraic manipulation is seen to be equivalent to
\begin{equation*}
  K \geq 2\; \dfrac{1}{1+\log(v_p)/\log(r)}+ \dfrac{\log(d/\nu)}{\log(rv_p)}.
\end{equation*}
Recalling that $v_p=(2p-1)/(2p)$ and $d_p=(4p^2-1)/(12p^2)$, the study of
the dependence of the bound of $K$ in $r$ yields two cases :
\begin{itemize}
\item if $\nu\leq 1/4$, that is $\max_p(v_p)\nu\leq\min_p(d_p)$, then $K=3$
  independently of $r$ and $p$, that is, if $\Delta x$ is fixed, independently
  of $\eps$ and $p$;
\item if $\nu\in[1/4,2]$, if one chooses $ r\leq d_p/\nu$, $K$ can be safely
  taken equal to $3$ as well.
\end{itemize}

Under these hypotheses, we conclude that the projective integration method has
an $\epsilon$-independent computational cost.

We illustrate this result numerically. We again consider a forward
Euler+centered flux formulation of the kinetic equation \eqref{e:kinetic} with
$\epsilon=1\cdot 10^{-2}$ in the velocity space \eqref{e:disc-vel} using
$p=10$ on the spatial mesh $\Pi:=\left\{x_0=-1+\Delta x/2,\ldots,1-\Delta
  x/2\right\}$ with $\Delta x=0.05$ and periodic boundary conditions. The time
step is $\delta t= \epsilon^2$. We plot the eigenvalues together with the
stability regions of the projective forward Euler method with $K=1,2,3$ and
$\Delta t = 2\Delta x^2/d_p$. The results are shown in figure
\ref{fig:pi-stab}. We see that, in this case, for which $r=\epsilon/\Delta
x=0.2$, $K=3$ inner steps are required to ensure overall stability.  Also note
that the stability region $\mathcal{D}_1^{PI}$ does not depend on $K$.
\begin{figure}
  \begin{center}
    \includegraphics[scale=0.9]{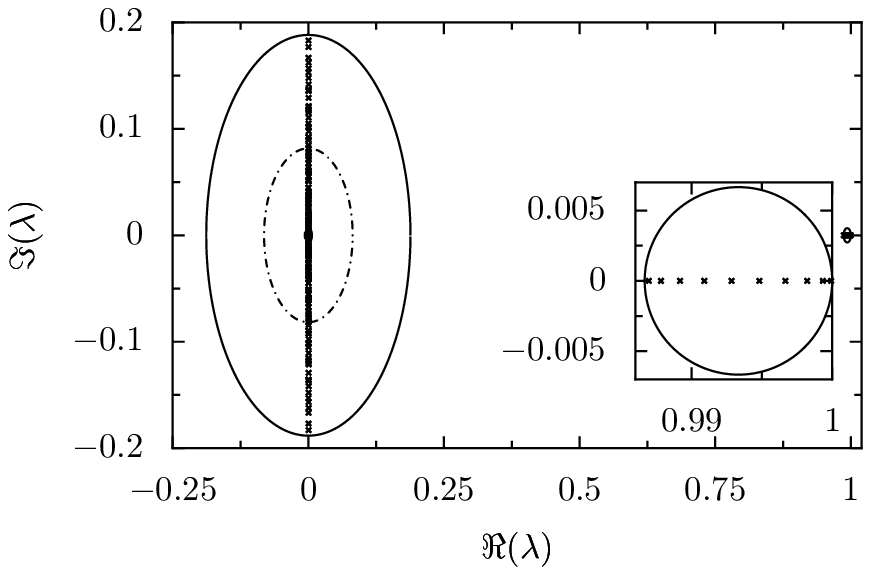}

  \end{center}
  \caption{\label{fig:pi-stab}Comparison of the eigenvalues of a forward
    Euler+centered flux of \eqref{e:kinetic} and the stability regions of the
    projective forward Euler method with $K=1$ (dashed), $K=2$ (dashdotted)
    and $K=3$ (solid). The inset is a zoom on the neighbourhood around $1$.  }
\end{figure}

\section{Consistency analysis\label{sec:consist}}

Our next goal is to estimate the consistency error in the macroscopic quantity
$\rho_\eps$ that is made at each outer step of the scheme. To this end, we
will study the truncation error in $f_\eps$, keeping in mind the Hilbert expansion
$f_\eps=\rho_\eps+\eps g_\eps$ \eqref{e:hilb}. Again, to simplify notations, 
the subscript $\eps$ is used only when dealing with the solution $f_\eps$ of
the continuous equation \eqref{e:discrete-kinetic} whereas $f$ is used for the
solution of the numerical scheme.  Let us recall the brute force
inner scheme \eqref{e:num_scheme}, in which we now consider vector-valued
quantities, hereby omitting the subscripts that refer to the dependence in
$x$ or $v$,
\begin{equation}
  f^{k+1}=S_{\delta t}f^k=f^k+\delta t\left( -\dfrac{\Phi(f^k)}{\eps}+\dfrac{\rho^k-f^k}{\eps^2}\right),\label{e:kinf}
\end{equation}
where $\Phi$ is the linear spatial discretization operator
$$\Phi(f)_{i,j}=\dfrac{\phi(f)_{i+1/2,j}-\phi(f)_{i-1/2,j}}{\Delta x}.$$  Following the stability condition in
Proposition \ref{prop:spectrum_fe}, we take $\delta t=\eps^2$ and we only
consider the centered flux \eqref{e:central_fe} ; the upwind flux case \eqref{e:upwind_fe} is commented on at the end of
the section. Note that, with this centered choice, $\Phi$ satisfies
\begin{equation}\langle\Phi(f^k)\rangle=\eps\langle
  \Phi(g^k)\rangle.\label{e:phi_mean}
\end{equation}  The inner
scheme then reduces to
\begin{equation}\label{e:kinf_eps2}
  f^{k+1}=S_{\eps^2}f^k=\rho^k-\eps\;\Phi(f^k).
\end{equation}
In spatial vector notation $\rho^N=(\rho_i^N)_{i\in\{0,\hdots,N_x\}}$, the
projective scheme \eqref{e:pfe} in $\rho$ reads
\begin{equation}
  \rho^{N+1}=\rho^{N,K+1}+(\Delta t-(K+1)\eps^2)\dfrac{\rho^{N,K+1}-\rho^{N,K}}{\eps^2},\label{e:prhoe}
\end{equation}
where $K+1$ is the number of small steps, and the values $\rho$ are obtained
as the averages over velocity space of the numerical approximations in
\eqref{e:kinf}.

 Let $f_\eps$ be a smooth solution of \eqref{e:discrete-kinetic}.
To bound the truncation error in $\rho$ of the projective integration method
\eqref{e:prhoe}, we introduce the following notations, $\forall N\geq 0$,
$k\in\{0,\hdots,K+1\}$, $(i,j)\in\{0,\hdots,N_x\}\times\J$ :
\begin{itemize}
\item the exact solution at time $t^{N,k}$, i.e.~$\tilde{f}^{N,k}_{\eps,i,j} :=
  f_{\eps,j}(x_i,t^{N,k})$;
\item the corresponding density
  $\tilde{\rho}^{N,k}_{\eps,i}:=\rho_\eps(x_i,t^{N,k})=\langle\tilde{f}_{\eps,i}^{N,k}\rangle$;
\item intermediate values obtained through iterations of the inner scheme, starting from the
  exact solution, $\bm{f}^{N,k}:=\left(S_{\eps^2}\right)^{k}\tilde{f}_\eps^N$;
\item and the corresponding density
  $\bm{\rho}^{N,k}:=\langle\bm{f}^{N,k}\rangle$.
\end{itemize}
Note that $\bm{f}^{N,0}=\tilde{f}_\eps^{N}$ and $\bm{\rho}^{N,0}=\tilde{\rho}_\eps^{N}$.
The truncation error in $\rho$ of the projective scheme \eqref{e:prhoe}, is
the quantity
\begin{equation*}
  \bm{E}^{N}: =
  \dfrac{\tilde{\rho}_\eps^{N+1}-\bm{\rho}^{N,K+1}}{\Delta t} 
  -\left(\dfrac{\Delta t-(K+1)\eps^{2}}{\Delta t}\right)
  \dfrac{\bm{\rho}^{N,K+1}-\bm{\rho}^{N,K}}{\eps^2}.
\end{equation*}

To bound $\bm{E}^{N}$ in the $L^2$ norm, we first estimate the iterated truncation errors of
\eqref{e:kinf} with respect to the equation \eqref{e:discrete-kinetic} for
$k\in\{0,\hdots,K+1\}$, given as
\begin{equation*}
  \bm{e}_f^{N,k}:=\dfrac{\tilde{f}_\eps^{N,k}-\bm{f}^{N,k}}{\eps^2}.
\end{equation*}
Using a Taylor expansion for the exact solution $f$ combined with equation
\eqref{e:discrete-kinetic}, and assuming that $f$ is such that
$\partial_{tt}^2f$ and $\partial_{x^k}^kf$, $k\in\{1,\hdots,K+1\}$ are bounded
uniformly with respect to $\eps$, a straightforward computation yields
\begin{equation}  \bm{e}_f^{N,k+1}=S_{\eps^2}\bm{e}^{N,k}+\dfrac{1}{\eps}(\Phi(\tilde{f}_\eps^{N,k})-v\partial_xf_\eps(t^{N,k}))+O(\eps^2).\label{e:trunc_e}
\end{equation}
We thus have, recalling that $\bm{e}^{N,0}=0$ and that the spectrum of
$S_{\eps^2}$ lies in the unit disk,
\begin{align}
 \bm{e}^{N,K+1}_f=\dfrac{1}{\eps}\sum_{k=0}^KS_{\eps^2}^{K-k}\left(\Phi(\tilde{f}_\eps^{N,k})-v\partial_xf_\eps(t^{N,k})\right)+O((K+1)\eps^2).\label{e:est}
\end{align}
\begin{rem}
  In the above formula, as well as in the remainder of the section, to keep
  the notations as clear as possible, the Landau symbol $O(\cdot)$ should be
  understood as an estimate in the $L^2$ norm.
\end{rem}
Unfortunately, if we compute directly the spatial truncation error in
\eqref{e:trunc_e}, the centered difference being of order 2, a stiff term
$\Delta x^2/\eps$ appears. We therefore proceed by estimating $\langle
S_{\eps^2}^k\Delta f\rangle$, where $\Delta :f_\eps\in
C_c^\infty(\T\times(0,T);\R^{2p})\mapsto\Phi(\tilde{f}_\eps)-v\widetilde{\left(\partial_xf_\eps\right)}\in\R^{2p\times
  N_x}$ where $\tilde{\ }$ is again the projection on the discretization
points. Note that $\Delta$ is a linear operator and, more precisely, that it
is the truncation error of the approximation of the first order spatial
differential by the centered scheme, so that $\Delta f_\eps=O(\Delta x^2)$. In
what follows, for the sake of simplicity, we will denote the composition
$\Phi\circ\Delta$ (resp. $S_{\eps^2}\circ\Delta$) by the product $\Phi\Delta$
(resp. $S_{\eps^2}\Delta$).

  The crucial first step is to see
that \eqref{e:kinf_eps2} reads
\begin{equation*}
  S_{\eps^2}\Delta f_\eps=\langle\Delta f_\eps\rangle-\eps\Phi\Delta f_\eps,
\end{equation*}
and, consequently,
\begin{equation*}
  S^2_{\eps^2}\Delta f_\eps=\langle\Delta f_\eps\rangle-\eps\left\{\langle
    \Phi\Delta f_\eps\rangle + \Phi\langle\Delta f_\eps\rangle\right\}+\eps^2\Phi^2\Delta f_\eps.
\end{equation*}
A simple combinatoric argument implies that, for $k\geq 3$,
\begin{align}
  S_{\eps^2}^k\Delta f_\eps&=\langle\Delta f_\eps\rangle-\eps\left\{\langle \Phi\Delta     f\rangle + \Phi\langle\Delta f_\eps\rangle\right\}\nonumber\\
  &+ \eps^2\left\{\langle\Phi^2\Delta f_\eps\rangle+\Phi^2\langle\Delta
    f\rangle+\Phi\langle\Phi\Delta
      f\rangle+(k-3)\langle\Phi^2\langle\Delta
      f\rangle\rangle\right\}+O(\eps^3).\nonumber
\end{align}
Taking the mean value of $S_{\eps^2}^k\Delta f_\eps$, and using \eqref{e:phi_mean}, as well as the fact that the linear operator $\Phi-v\partial_x$ is odd in $v$, we get :
\begin{itemize}
\item for $k=0$ : $\langle\Delta f_\eps\rangle=\eps \langle\Delta g_\eps\rangle=\eps
  O(\Delta x^2)$,
\item for $k=1$ : $\langle S_{\eps^2}\Delta f_\eps\rangle=\eps \left\{\langle\Delta g_\eps\rangle-\langle\Phi\Delta \rho_\eps\rangle\right\}-\eps^2\langle\Phi\Delta g_\eps\rangle =\eps O(\Delta x^2)+\eps^2O(\Delta x^2)$,
\item for $k\geq 2$ :
\begin{align}
  \langle S_{\eps^2}^k\Delta f_\eps\rangle&= \eps\left\{\langle\Delta g_\eps\rangle-\langle\Phi\Delta \rho_\eps\rangle\right\}
  +\eps^2\left\{\langle\Phi^2\Delta \rho_\eps\rangle-\langle\Phi\Delta g_\eps\rangle\right\}+O(\eps^3),\nonumber\\
  &=\eps O(\Delta x^2)+\eps^2 O(\Delta x^2)+O(\eps^3) \nonumber.
\end{align}
\end{itemize}
Using Young's inequality and plugging this estimate in \eqref{e:est}, we get
\begin{align}
  \bm{e}^{N,k+1}&=O((k+1)\Delta x^2)+O((k+1)\eps^2),\,\,\forall k\geq 1,
\end{align}
and, consequently, 
\begin{align}
  \bm{E}^{N}&=\dfrac{\tilde{\rho}_\eps^{N+1}-\tilde{\rho}_\eps^{N,K+1}+\eps^2\langle\bm{e}^{N,K+1}\rangle}{\Delta
    t}-\left(\dfrac{\Delta t-(K+1)\eps^2}{\Delta t}\right)
  \dfrac{\tilde{\rho}_\eps^{N,K+1}-\tilde{\rho}_\eps^{N,K}}{\eps^2}\nonumber\\
  &\hspace*{1cm}-	\left(\dfrac{\Delta t-(K+1)\eps^2}{\Delta
    t}\right)
  \langle\bm{e}^{N,K+1}-\bm{e}^{N,K}\rangle\nonumber\\
  &=\left(1-(K+1)\dfrac{\eps^2}{\Delta
      t}\right)\partial_t\rho(t^{N,K+1})-\left(1-(K+1)\dfrac{\eps^2}{\Delta t}\right)\partial_t\rho(t^{N,K+1})\nonumber\\
  &\hspace*{1cm}+O(\Delta t)+\eps^2O\left(\dfrac{\Delta x^2+\eps^2}{\Delta
      t}\right)+O(\eps^2)+O(\Delta x^2).\nonumber
\end{align}
Following the classical parabolic CFL condition $\Delta t= O(\Delta
x^2)$, we get
\begin{align}
  \bm{E}^N&= O(\Delta t)+O\left(\dfrac{\eps^4}{\Delta t}\right)+O(\eps^2).\nonumber
\end{align}
The projective scheme is consistent with \eqref{e:diffeps} at order 2 in
$\eps$ and, as $\eps$ goes to $0$, the limiting scheme is consistent at order
1 in time and 2 in space with \eqref{e:diff}.

\begin{rem}[Upwind fluxes]
 For the upwind flux, the fact that $\langle \Phi(\tilde{\rho}_\eps)\rangle=O(\Delta
    x^2)$, instead of vanishing as in the centered flux case, implies a consistency error term of order $\Delta x^2/\eps$ that is not easily
    cancelled. 
\end{rem}
\begin{rem}[Hilbert expansion]
When taking $\delta t=\eps^2$, rewriting the scheme \eqref{e:kinf} in
    terms of $\rho$ and $g=(f-\rho)/\eps$ leads to
    \begin{align}
      \rho^{k+1}_i&=\rho_i^k-\eps^2\langle\Phi_i(g^k)\rangle\nonumber\\
      g^{k+1}_{i,j}&=-\Phi_{i}(\rho^k)-\eps\left\{\Phi_{i,j}(g^k)-\langle\Phi_i(g^k)\rangle\right\}.\nonumber
    \end{align}
    which is a scheme for \eqref{e:kin_expanded}.  From this equation, the
    effect of the particular choice of time step $\delta t=\eps^2$ becomes
    clear.  With this time step, in accordance with \eqref{e:hilb} and
    \eqref{e:g}, we see that the distribution satisfies
    \begin{equation*}
      f^N=\rho^N+\eps g^N = \rho^N -\eps \Phi(\rho^N)+O(\eps^2),
    \end{equation*}
    and, therefore, the projective integration scheme recovers the first two
    terms in the Hilbert expansion of $f$.
\end{rem}
In conclusion, we summarize the above results on stability, consistency and
the number of steps :
\begin{thm}\label{t:thm}
  Under the CFL condition $\Delta t= \Delta x^2/(4d_p)$, for $K\geq 3$, assuming
  $\partial^2_{tt}f$ and $\partial_{x^k}^kf$, $k\in \{1,\ldots, K+1\}$, are bounded uniformly with respect to
  $\eps$, the following estimate holds for all $N\Delta t\leq T$,
  \begin{equation}
    \|\rho_\eps(t^N)-\rho^N\|_2=T \left(O(\Delta t)+O\left(\dfrac{\eps^4}{\Delta
          t}\right)+O(\eps^2)\right).\label{e:cvg}
  \end{equation}
\end{thm}

The estimate \eqref{e:cvg} shows that, if $\eps$ is fixed, the optimal choice
of $\Delta t$ in terms of accuracy is $\Delta t=O(\eps^2)$. Of course, this leads
to prohibitive costs as $\eps\to 0$, but it allows us to consider a solute
computed with this choice of time-step to be precise at order $\eps^2$.
Larger values of $\Delta t$, and in particular the values $\Delta t=O(\Delta
x^2)$ that we envision, increase the error due to the outer step; smaller
values increase the error due to the time derivative estimation from the inner
steps.  Note also that taking $\Delta t \to 0$ for fixed $\eps$ would
completely defeat the purpose of the projective integration
method. Additionally, in this limit the method is unstable, since this choice
implies $\Delta x \to 0$ for fixed $\eps$ (see also remark \ref{rem:dx}).


\section{Numerical results\label{sec:num}}

In this section, we illustrate the above results via the numerical simulation of
two model problems, namely the linear kinetic equation \eqref{e:kinetic}, and
the Su-Olson problem \eqref{e:su-olson}.

\subsection{Linear kinetic equation}

We consider equation \eqref{e:kinetic} on the velocity space
\eqref{e:disc-vel} using $p=10$ with $\epsilon=5\cdot 10^{-2}$ on the spatial
mesh $\Pi:=\left\{x_0=-1+\Delta x/2,\ldots,1-\Delta x/2\right\}$ with $\Delta
x=0.1$ and periodic boundary conditions.  As an initial condition, we take
\begin{equation*}
  f_\epsilon(x,v,t=0) = 
  \begin{cases}
    2,\qquad \text{ for } -− 0.5\le x  \le 0.5 \text{ and } -− 0.75 \le v \le 0.25,\\
    1,\qquad \text{ otherwise}.
  \end{cases}
\end{equation*}
We define the initial condition for the recursion \eqref{e:num_scheme} by
taking cell averages, i.e.
\[
f^0_{i,j}=\int_{x_i-\Delta x/2}^{x_i+\Delta x/2}f_\epsilon(x,v,t=0)dx.
\]
We perform a time integration up to time $t=2.5$ using a centered flux/forward
Euler scheme with $\delta t=\epsilon^2$, as well as a projective forward Euler
integration, again using $\delta t = \epsilon^2$, and additionally specifying
$K=4$ and $\Delta t = \nu \Delta x^2/d_p$, with $d_p=\langle v^2\rangle=0.3325$
(using the cell averages) and $\nu=1$. For comparison purposes, we also
compute the result using a centered flux/forward Euler scheme with $\delta
t=\epsilon^3$, which we will consider to be the ``exact'' solution, and a
solution of the limiting heat equation on the same mesh using $\Delta t =
0.4\Delta x^2/d_p$.  The experiment is repeated for $\epsilon=2\cdot 10^{-2}$
and $\Delta x=0.1$ and $\Delta x=0.05$, respectively.  The results are shown
in figure \ref{fig:sol}. We show both $\rho_\eps(x,t=2.5)$ and the flux
$J_\eps(x,t=2.5)$, with
\[
J_\eps(x,t)=\dfrac{1}{\eps}\int_{\V}vf_\epsilon(x,v,t)dv.
\]
We see that the complete simulation with $\delta t=\epsilon^2$ and $\delta
t=\epsilon^3$ visually coincide in all cases.  The projective integration
method results in a solution that is closer to that of the limiting heat
equation.  Note that the differences between all solutions become smaller for
decreasing $\epsilon$ and $\Delta x$. (For $\Delta x=0.1$, the difference
between projective integration and the ``exact'' solution are mainly due to
the space, and correspondingly large time step, in accordance with theorem \ref{t:thm}.)

\begin{figure}
  \begin{center}
    \includegraphics[scale=0.8]{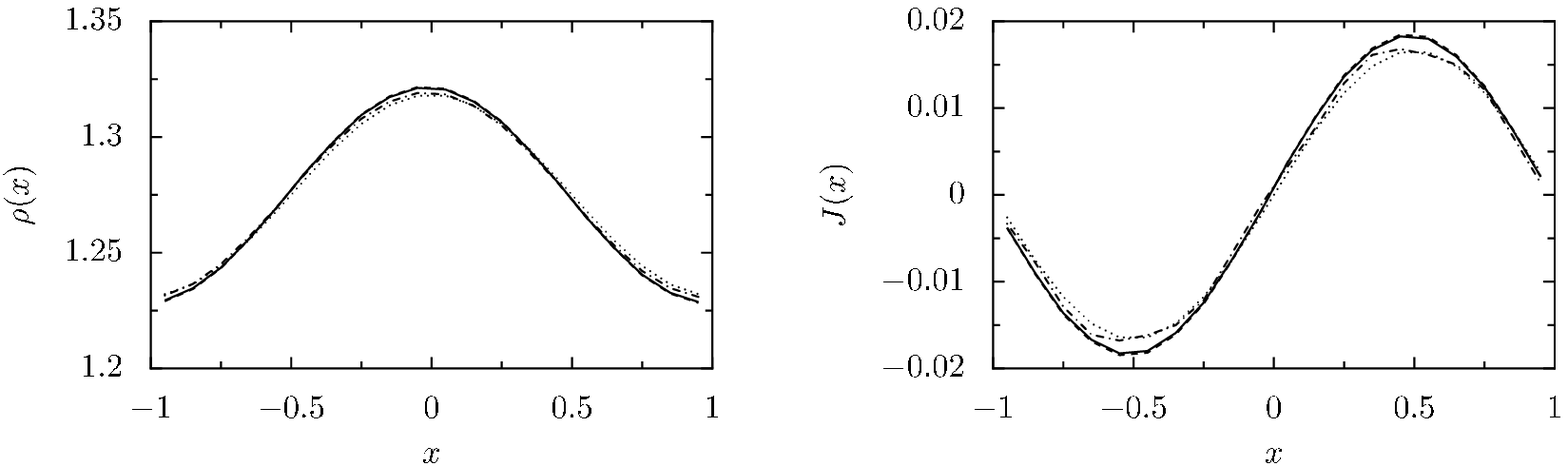}
    \includegraphics[scale=0.8]{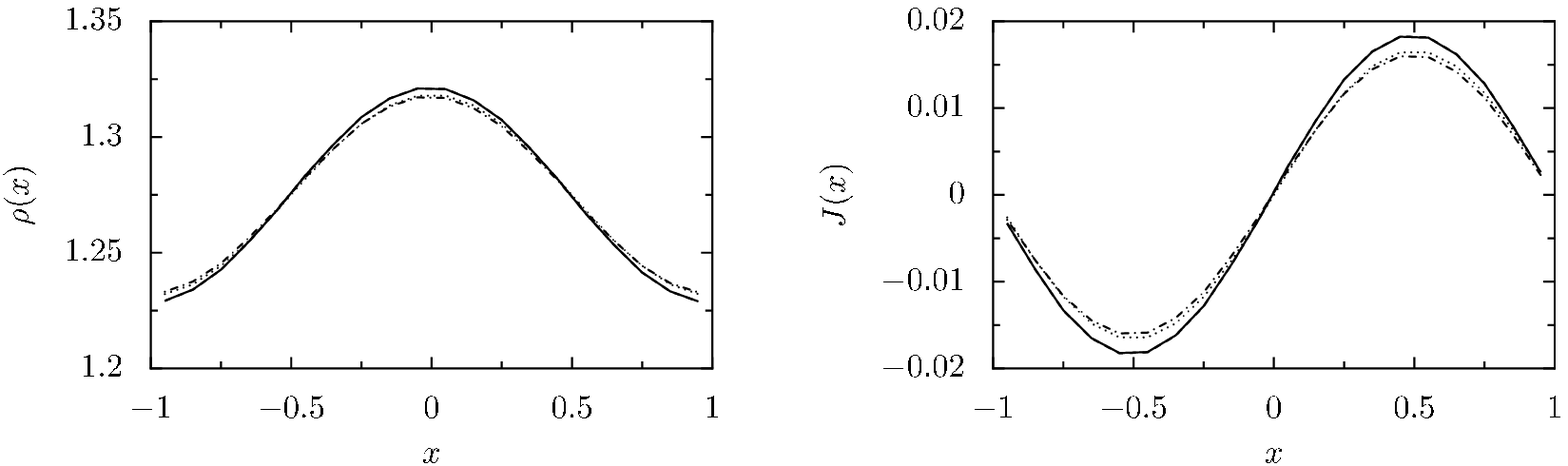}
    \includegraphics[scale=0.8]{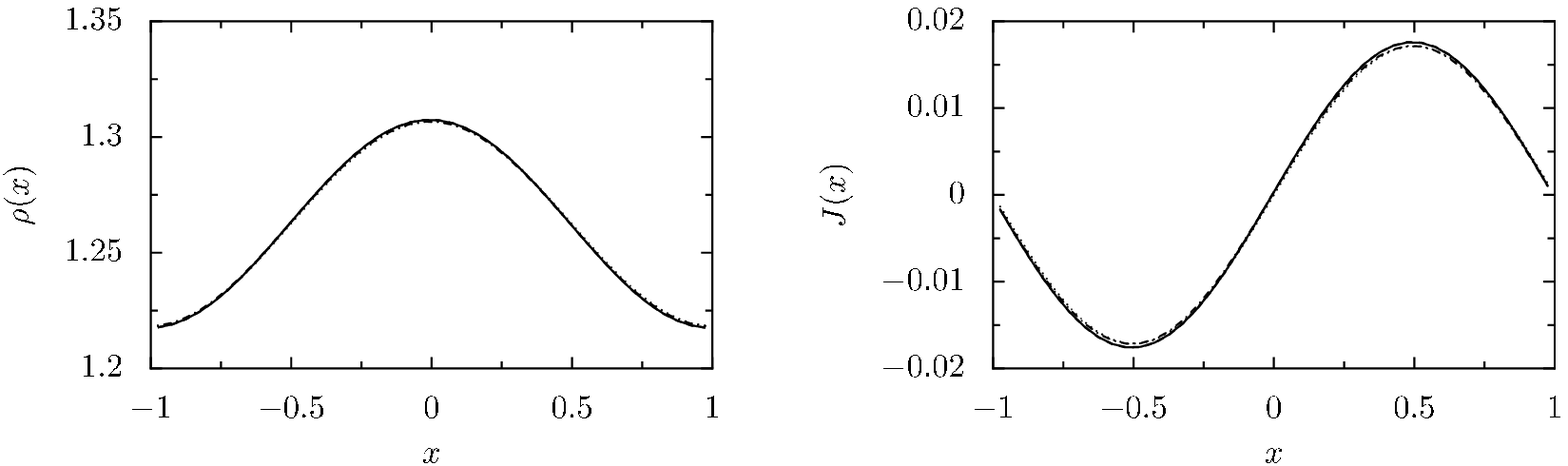}
  \end{center}
  \caption{\label{fig:sol} Results of simulation of the kinetic equation
    \eqref{e:kinetic} at time $t=2.5$. Left: density; right: flux.  Parameter
    values are (a) $\Delta x=0.1$ and $\epsilon=0.05$ (top); (b) $\Delta
    x=0.1$ and $\epsilon=0.02$ (middle); and (c) $\Delta x = 0.05$ and
    $\epsilon=0.02$.  Shown are (i) a centered flux/forward Euler flux scheme
    with $\delta t=\epsilon^2$ (solid); (ii) a centered flux/forward Euler
    scheme with $\delta t=\epsilon^3$ (dashed); and (iii) the projective
    integration method (dashdot). For comparison, also the solution of the
    heat equation is shown (dot).}
\end{figure}

Next, we look at the convergence properties in terms of $\epsilon$.  To this
end, we repeat the computation for $\Delta x=0.1$ and different values of
$\epsilon$.  We consider the centered flux/forward Euler flux finite volume
scheme with $\delta t = \epsilon^3$ to be the ``exact'' solution and
approximate the error of the other simulations by the difference with respect
to this reference simulation. (This reference solution is at least an order in
$\epsilon$ more accurate than the other results.) We first investigate the
error of the centered flux/forward Euler flux finite volume scheme with
$\delta t=\epsilon^2$.  The results are shown in figure \ref{fig:eps} (top).
The $\mathcal{O}(\epsilon^2)$ behaviour is apparent.
\begin{figure}
  \begin{center}
    \includegraphics[scale=0.9]{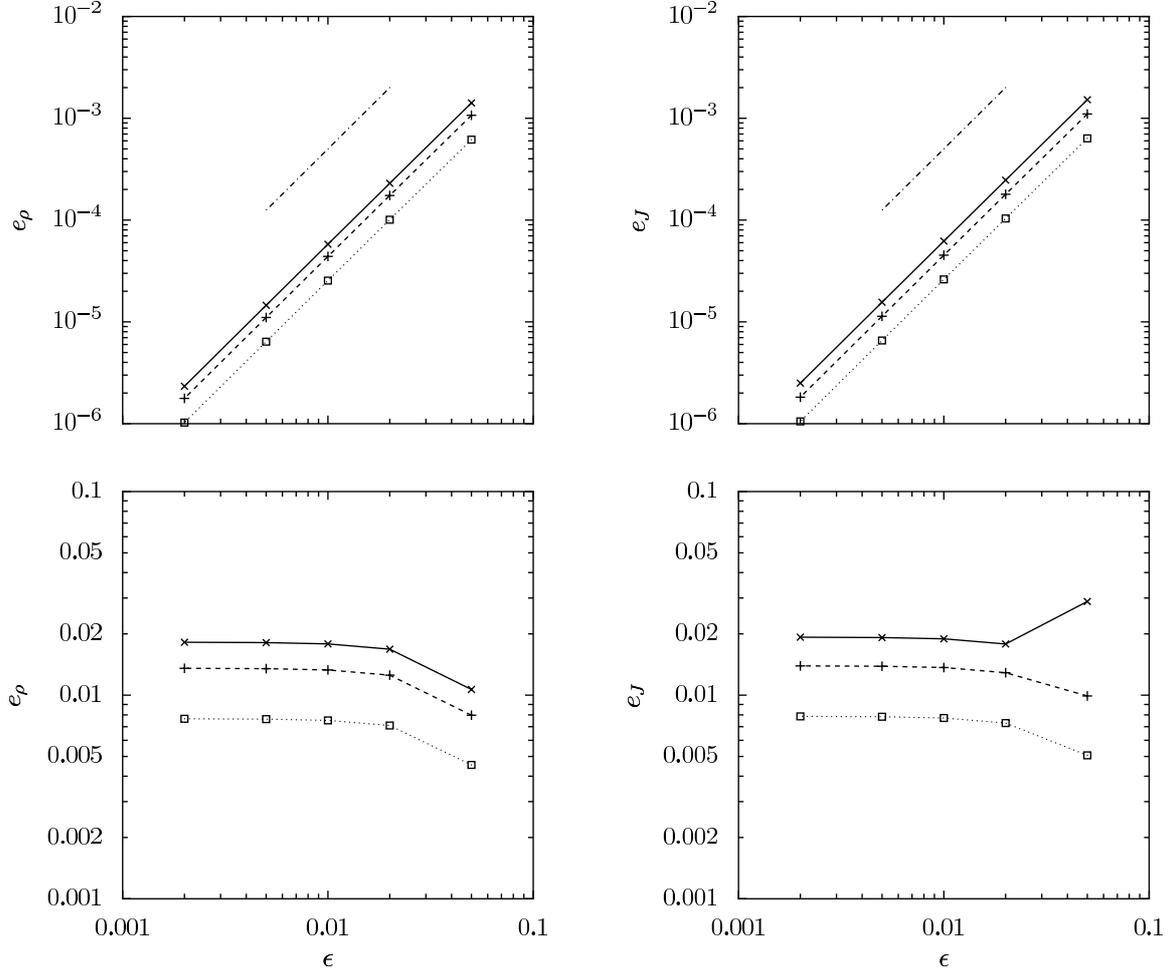}
  \end{center}
  \caption{\label{fig:eps}Top: Error in density (left) and flux (right) of a
    centered flux/forward Euler flux integration of the kinetic equation
    \eqref{e:kinetic} using $\Delta x=0.1$ and $\delta t=\epsilon^2$ as a
    function of $\epsilon$ at time $t=1.25$ (dot), $t=2.5$ (dash) and $t=3.75$
    (solid).  For comparison, we also plot a line with slope $2$ (dashdot),
    indicating the error that is predicted by the consistency
    analysis. Bottom: Error in density (left) and flux (right) of a projective
    forward Euler integration of the kinetic equation \eqref{e:kinetic} using
    $\Delta x=0.1$, $\delta t=\epsilon^2$, $K=3$ and $\Delta t=\Delta x^2/d_p$
    as a function of $\epsilon$ at time $t=1.25$ (dot), $t=2.5$ (dash) and
    $t=3.75$ (solid). The error is $O(\Delta x^2)$ was predicted by theorem \ref{t:thm}. }
\end{figure}
In the same way, we consider the error of projective integration using $K=3$
and $\Delta t = \Delta x^2/d_p$. Figure \ref{fig:eps} shows that this error is
largely independent of $\epsilon$, especially when $\epsilon\to 0$.

Finally, we look at the error of projective forward Euler as a function of
$\Delta t$. We perform a projective forward Euler simulation using $\Delta
x=0.1$ and $\epsilon=2\cdot 10^{-3}$ using $\delta t =\epsilon^2$ and
$K=3$. As the projective step size, we use $\Delta t=\nu\Delta x^2/d_p$ for a
range of values of $\nu$, and we again compare the density and flux with
respect to the reference solution.  The results are shown in figure
\ref{fig:Dt}. We clearly see the first order behaviour as $\Delta t\to 0$.
Also remark that, for large values of $\Delta t$, the error increases quickly
due to a loss of stability.  From this figure, it can be checked that the loss
of stability indeed occurs at $\nu=2$, as discussed in section \ref{sec:pi}, whereas for the limiting heat equation,
$\nu=0.5$ is the maximal value that should be used to ensure stability.
\begin{figure}
  \begin{center}
    \includegraphics[scale=0.9]{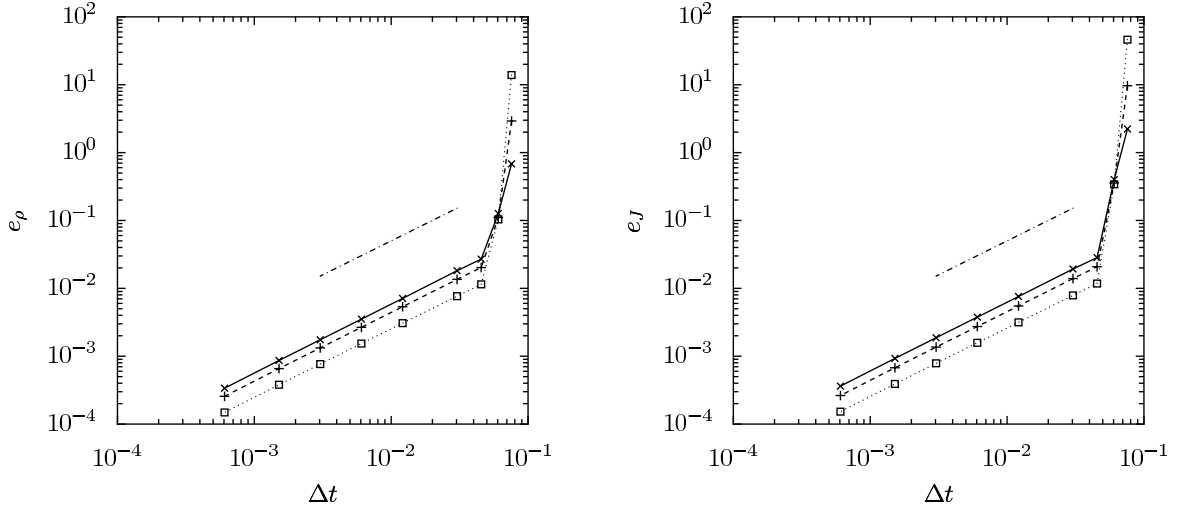}
  \end{center}
  \caption{\label{fig:Dt}Error in density (left) and flux (right) of a
    projective forward Euler integration of the kinetic equation
    \eqref{e:kinetic} with $\epsilon=2\cdot 10^{-3}$ using $\Delta x=0.1$,
    $\delta t=\epsilon^2$, $K=3$ as a function of $\Delta t$ at time $t=1.25$
    (dot), $t=2.5$ (dash) and $t=3.75$ (solid).  For comparison, we also plot
    a line with slope $1$ (dashdot), indicating the error that is predicted by
    the consistency analysis.  }
\end{figure}

\subsection{Su-Olson problem}

We consider equation \eqref{e:su-olson} on the velocity space
\eqref{e:disc-vel} using $p=10$ with $\epsilon=5\cdot 10^{-2}$ on the spatial
mesh $\Pi:=\left\{x_0=-1+\Delta x/2,\ldots,30-\Delta x/2\right\}$ with $\Delta
x=0.1$ and homogeneous Neumann boundary conditions.  As an initial condition,
we take
\begin{equation*}
  f_\epsilon(x,v,t=0) = \theta_\eps(x,t=0) = A,
\end{equation*}
with $A=1$ and $A=1\cdot 10^{-10}$, respectively.  Again, we take the cell
averages of the initial condition, and perform a time integration up to time
$t=1$ using (i) a centered flux/forward Euler flux scheme with $\delta
t=\epsilon^2$, and (ii) a projective forward Euler integration, again using
$\delta t = \epsilon^2$, and additionally specifying $K=3$ and $\Delta t = \nu
\Delta x^2/d_p$, with $d_p=\langle v^2\rangle=0.3325$ (using the cell-averaging)
and $\nu=1$. For comparison purposes, we also compute a reference solution
using a centered flux/forward Euler flux scheme with $\delta t=\epsilon^3$.
The results are shown in figure \ref{fig:sol-suolson}.
\begin{figure}
  \begin{center}
    \includegraphics[scale=0.6]{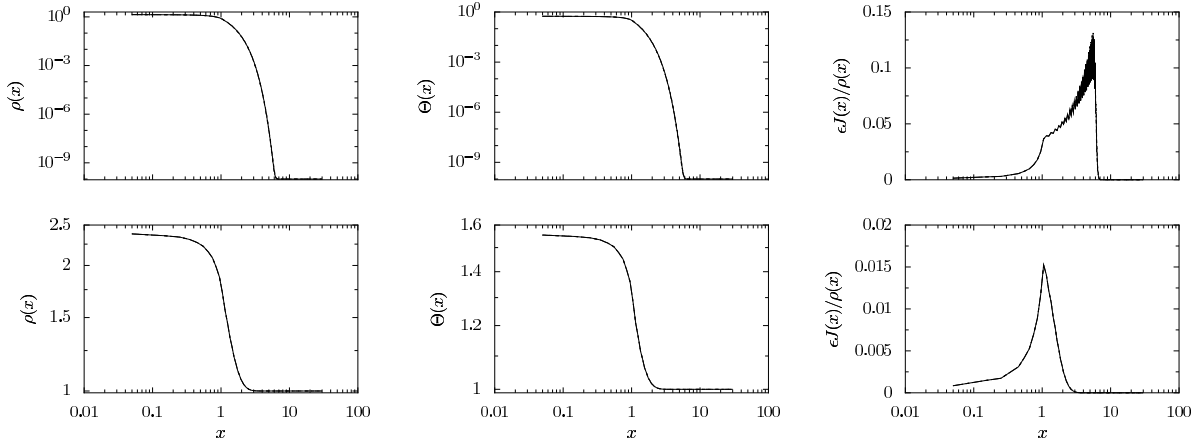}
  \end{center}
  \caption{\label{fig:sol-suolson} Results of simulation of the kinetic
    equation \eqref{e:su-olson} at time $t=1.$ using $A=1\cdot 10^{-10}$ (top)
    and $A=1$ (bottom). Left: $\rho_\eps(x,t=1)$; middle:
    $\theta_\eps(x,t=1)$; right: $\epsilon
    J_\eps(x,t=1)/\rho_\eps(x,t=1)$. Shown are (i) a centered flux/forward
    Euler flux scheme with $\delta t=\epsilon^2$ (solid); (ii) a centered
    flux/forward Euler scheme with $\delta t=\epsilon^3$ (dashed); and (iii)
    the projective integration method (dashdot). All simulations visually
    coincide.  }
\end{figure}
We show $\rho_\eps(x,t=1)$ and $\theta_\eps(x,t=1)$, as well as the quantity
$\epsilon J_\eps/\rho_\eps$, which is a measure of the ``limited-flux
property'', that is, if $f_\eps$ is nonnegative, we should have
\begin{equation*}
  \eps|J_\eps|= \left|\dfrac{1}{2p}\sum_{j\in \J}v_jf_\eps\right|\leq 
  \dfrac{1}{2p}\sum_{j\in \J}|v_j|f_\eps\leq \|v\|_\infty \rho_\eps.
\end{equation*}
Only the result of projective integration is visible, since the solutions
using the other procedures are visually indistinguishable on this scale. The
errors of the full forward Euler simulation and the projective integration are
shown in figure \ref{fig:sol-suolson-error}.
\begin{figure}
  \begin{center}  
    \includegraphics[scale=0.6]{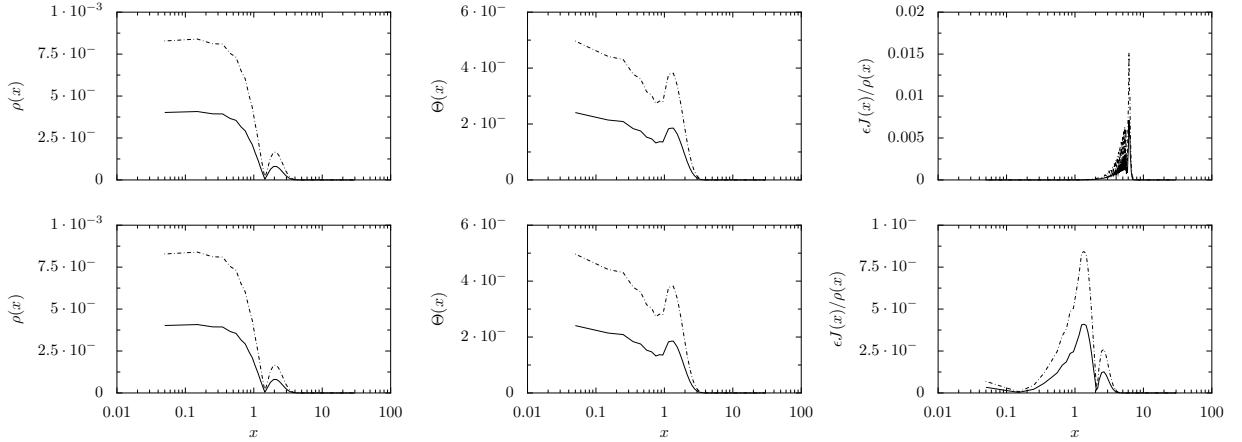}
  \end{center}
  \caption{\label{fig:sol-suolson-error} Error of the numerical solution of
    the kinetic equation \eqref{e:su-olson} at time $t=1.$ using $A=1\cdot
    10^{-10}$ (top) and $A=1$ (bottom). Left: $\rho_\eps(x,t=1)$; middle:
    $\theta_\eps(x,t=1)$; right: $\epsilon
    J_\eps(x,t=1)/\rho_\eps(x,t=1)$. Shown are the error with respect to a
    reference solution of (i) a centered flux/forward Euler flux scheme with
    $\delta t=\epsilon^2$ (solid); and (ii) the projective integration method
    (dashdot). }
\end{figure}
We clearly see that, while the computational cost of projective integration is
much lower, the error remains of the same order of magnitude.

\section{Conclusions and discussion\label{sec:concl}}

We investigated a projective integration scheme for the numerical solution of
a kinetic equation in the limit of small mean free path, in which the kinetic
description approaches a diffusion equation. The scheme first takes a few
small inner steps with a simple, explicit method, such as a centered flux/forward
Euler finite volume scheme, and subsequently
extrapolates the results forward in time over a large outer time step on
the diffusion time scale. We provided a stability and consistency result,
showing that the method is asymptotic-preserving.

We conclude with some remarks, and some directions for future results.  First,
a higher order outer integration method can readily be used to obtain a higher
order in the macroscopic time step $\Delta t$, see
e.g.~\cite{RicoGearKevr04,Lee2007p2355}. We emphasize that this higher order
accuracy does not depend on the order of the inner simulation, since the error
in time at that level is of the order $\mathcal{O}(\eps^2)$, due to the choice
of the inner time step.

Several new directions are currently being pursued. First, we are extending
these results to the kinetic Fokker--Planck case, which requires a precise
study of the discretization of the second-order derivation in the velocity
variable.  Furthermore, we are also considering projective integration in
conjunction with a relaxation method \cite{arna} to obtain a general method
for nonlinear systems of hyperbolic conservation laws in multiple dimensions.

\section*{Acknowledgements}

GS is a Postdoctoral Fellow of the Research Foundation -- Flanders (FWO --
Vlaanderen). This work was performed during a research stay of GS at SIMPAF
(INRIA - Lille), whose hospitality is gratefully acknowledged. PL warmly
thanks T. Goudon and
B. Beckermann for fruitful discussions. The work of GS was supported by the
Research Foundation -- Flanders through Research Project G.0130.03 and by the
Interuniversity Attraction Poles Programme of the Belgian Science Policy
Office through grant IUAP/V/22. The scientific responsibility rests with its
authors.

\bibliographystyle{plain} \bibliography{scip}
\end{document}